\documentclass[12pt]{article}   	% use "amsart" instead of "article" for AMSLaTeX format
\usepackage{geometry}                		% See geometry.pdf to learn the layout options. There are lots.
\usepackage[parfill]{parskip}    		% Activate to begin paragraphs with an empty line rather than an indent
\usepackage{graphicx}				% Use pdf, png, jpg, or eps§ with pdflatex; use eps in DVI mode
\usepackage[pdftex,pagebackref,letterpaper=true,colorlinks=true,pdfpagemode=none,urlcolor=blue,linkcolor=blue,citecolor=blue,pdfstartview=FitH]{hyperref}
\usepackage{amsmath,amsfonts}
\usepackage{color}
\usepackage{amsthm}
\usepackage{boolexpr}
\usepackage{amssymb,latexsym,pstcol,pst-plot}
\usepackage[UKenglish]{isodate}% http://ctan.org/pkg/isodate
\usepackage{amssymb}
\allowdisplaybreaks

\newtheorem{thm}{Theorem}
\newtheorem{cor}[thm]{Corollary}

\title{Asymptotic density of Catalan numbers modulo 3 and powers of 2}
\author{Rob Burns}
\begin{document}
\maketitle
\begin{abstract}
We establish the asymptotic density of the Catalan numbers modulo $3$ and modulo $2^k$ for $k~\in~\mathbb{N}$ and $k~\geq~1$.
\end{abstract}

\section{Introduction}
%\section{}
%\subsection{}
The {\it Catalan numbers\/} are defined by
$$
C_n := \frac{1}{n+1}\binom{2n}{n}.
$$
There has been much work in recent years and also going back to Kummer \cite{Kum1852} on analysing the Catalan numbers modulo primes and prime powers.  Deutsch and Sagan \cite{Sagan2006} provided a complete characterisation of Catalan numbers modulo 3. A characterisation of the Catalan numbers modulo 2 dates back to Kummer. Eu, Liu and Yeh \cite{Eu2008} provided a complete characterisation of Catalan numbers modulo 4 and 8. This was extended by Liu and Yeh \cite{Liu2010} to a complete characterisation modulo 8, 16 and 64. This result was restated in a more compact form by Kauers, Krattenthaler and M\"uller in \cite{KKM2012} by representing the generating function of $C_n$ as a polynomial involving a special function. The polynomial for $C_n$ modulo 4096 was also calculated. A method for extracting the coefficients of the generating function (i.e. $C_n$ modulo a prime power) was provided, though given the complexity of the polynomials (the polynomial for the $4096$ case takes a page and a half to write down) this would need to be done by computer. Krattenthaler and M\"uller  \cite{KM2013} used a similar method to examine $C_n$ modulo powers of $3$. They wrote down the polynomial for the generating function of $C_n$ modulo $9$ and $27$ thereby generalised the mod $3$ result of \cite{Sagan2006}. The article by Lin \cite{Lin:2010ab} discussed the possible values of the odd Catalan numbers modulo $2^k$ and Chen and Jiang \cite{Chen2013} dealt with the possible values of the Catalan numbers modulo prime powers. 
Rowland and Yassawi \cite{RY2013} investigated $C_n$ in the general setting of automatic sequences.  The values of $C_n$ (as well as other sequences) modulo prime powers can be computed via automata. Rowland and Yassawi provided algorithms for creating the relevant automata. They established a full characterisation of $C_n$ modulo $\{2, 4, 8, 16, 3, 5 \}$ in terms of automata. They also extended previous work by establishing forbidden residues for $C_n$ modulo $\{32, 64, 128, 256, 512 \}$. In theory the automata can be constructed for any prime power but computing power and memory quickly becomes a barrier.

Some of this work can be used to determine the asymptotic densities of the Catalan numbers modulo $2^k$ and $3$.

In section~\ref{mod2^k} we will use a result of Liu and Yeh \cite{Liu2010} to obtain some asymptotic densities of the Catalan numbers modulo powers of $2$. Here, the asymptotic density of a subset $S$ of $\mathbb{N}$ is defined to be
$$
\lim_{N \to \infty} \frac {1}{N} \#\{ n \in S : n \leq N \}
$$
if the limit exists, where $\#S$ is the number of elements in a set $S$. Since the Catalan numbers $C_n$ are highly multiplicative as $n$ increases, it is expected that for a fixed $m \in \mathbb{N}$ ``almost all'' Catalan numbers are divisible by $m$. We show that this is the case when $m = 3$ and $m = 2^k$ for $k \geq 1$.

\section{Catalan numbers modulo $2^k$}
\label{mod2^k}
Firstly, as in \cite{Liu2010}, let the $p$-adic order of a positive integer $n$ be defined by
$$
\omega_p(\, n )\, := max\{\alpha \in \mathbb{N} : p^{\alpha} | n\} 
$$ 
and the cofactor of $n$ with respect to $p^{ \omega_p(\, n )\, } $  be defined by 
$$
CF_p(\, n )\,  := \frac {n} {p^{\omega_p(\, n )\,}}.
$$ 
In addition the function $\alpha : \mathbb{N} \to \mathbb{N}$  is defined by 
$$
\alpha (\, n )\, := \frac {CF_2(\, n + 1 )\, - 1} {2}.
$$
For a number $p$, we write the base $p$ expansion of a number $n$ as 
$$
[\, n ]\,_{p} = \langle n_{r} n_{r-1} ... n_{1} n_{0} \rangle
$$ 
where $n_{i} \in [\, 0, p - 1]\,$  and 
$$
n = n_{r} p^{r} + n_{r-1} p^{r-1} + ... + n_{1} p + n_{0}.
$$ 
Then the function $d_p : \mathbb{N} \to \mathbb{N}$ is defined by
\begin{equation}
\label{d_p}
d_p(\, n )\, := \#\{i: n_{i} = 1\}.
\end{equation}  
Here we will only be interested in the case $p = 2$ and will refer to $d_2$ as merely $d$. When $p=2$, $d(n)$ is the sum of the digits in the base 2 representation of $n$.

The following result appears in \cite{Liu2010}.
\bigskip

\begin{thm}[\emph{Corollary 4.3 of \cite{Liu2010}} ]
\label{cor4.3}
In general, we have $\omega_{2} (\, C_{n} )\, = d (\, \alpha (\, n )\, )\,$   In particular,  
$C_{n} \equiv 0 \mod  2^{k}$  if and only if  $d (\, \alpha (\, n )\, )\, \geq k$,  
and  
$C_{n} \equiv 2^{k - 1} \mod   2^{k}$  if and only if  $d (\, \alpha (\, n )\,  )\, = k - 1.$
\end{thm}

\bigskip

Since we are looking at the Catalan numbers $ C_n \mod 2^k$ it will be convenient to first consider numbers $n < 2^t$ for a fixed but arbitrary $t \in \mathbb{N}$. This produces some interesting and simple formulae. We have the following results:
\bigskip

\begin{thm}
\label{a}
$\#\{n < 2^t: d(\, \alpha (\, n )\, )\, = k \} = \binom {t} {k + 1}.$
\end{thm}
\begin{proof}
Let  
$$
\alpha = \alpha (\, n )\, = \frac {CF_2(\, n + 1 )\, - 1} {2}.
$$
Then $n + 1 = 2^{s} (\, 2 \alpha + 1 )\, $ for some arbitrary $s \in \mathbb{N} : s  \geq 0 $ .  Writing $n+1 $  in base $2$ we have
$$
[\, n + 1 ]\,_{2} = \langle [\, \alpha ]\,_{2}, 1, 0 ... 0, 0 \rangle 
$$  
where there are $s$  $0$'s at the end and $s \geq 0 $  is arbitrary. So,
$$
[\, n ]\,_{2} = \langle [\, \alpha ]\,_{2}, 0, 1 ... 1, 1 \rangle 
$$  
where there are $s$  $1$'s  at the end and $s \geq 0 $  is arbitrary. It can be seen that $d(\, n )\, = d(\, \alpha )\, + s. $
Since $n < 2^{t} $  and  $d(\, \alpha (\, n )\, )\, = k $  the possible base $2$ representations of $n$ are
$$
[\, n ]\,_{2} = \langle [\, \alpha ]\,_{2}, 0 \rangle :  \alpha < 2^{t - 1}
$$
$$
[\, n ]\,_{2} = \langle [\, \alpha ]\,_{2}, 0, 1 \rangle : \alpha < 2^{t - 2}
$$
$$ . $$
$$ . $$
$$
[\, n ]\,_{2} = \langle [\, \alpha ]\,_{2}, 0, 1, 1, ... 1 \rangle : \alpha < 2^{k} 
$$  
and there are $(\, t - k - 1)\, $   $1$'s at the end of the last representation.

Therefore, counting each of these possibilities and using the fact that $d (\, \alpha )\, = k $  gives
$$
\#\{n < 2^t: d(\, \alpha (\, n )\, )\, = k \} = \sum_{i = k}^{t - 1} \binom{i} {k} = \binom{t} {k + 1}.
$$
\end{proof}

\bigskip

\begin{cor}
\label{b}
$\#\{n < 2^t: C_{n} \equiv 0 \mod {2^k} \} =  \sum_{i = k+1}^{i = t} \binom{t} {i}.$
\end{cor}
\begin{proof}
The corollary follows from Theorem~\ref{cor4.3} and Theorem~\ref{a} above since
$$
\#\{n < 2^t: C_{n} \equiv 0 \mod {2^k} \}
$$
$$
\mbox{$= \#\{n < 2^t: d(\, \alpha (\, n )\,  )\,  \geq k \}$   from Theorem~\ref{cor4.3}}
$$
$$
= 2^{t} - 1 -  \sum_{i = 0}^{k - 1} \#\{n < 2^t: d(\, \alpha (\, n )\,  )\,  = i \}
$$
$$
\mbox{$= 2^{t} - 1 - \sum_{i = 0}^{k - 1} \binom{t} {i + 1} $  from Theorem~\ref{a}}
$$
$$
= 2^{t} - \sum_{i = 0}^{k} \binom{t} {i} 
$$
$$
\mbox{$= \sum_{i = k + 1}^{t} \binom{t} {i} $  since  $2^{t} = \sum_{i = 0}^{i = t} \binom{t} {i}$}
$$
\end{proof}
\bigskip
\begin{cor}
\label{c}
$\#\{n < 2^t: C_{n} \equiv 2^{k-1} \mod {2^k} \} =  \binom{t} {k}.$
\end{cor}
\begin{proof}
From Theorem~\ref{cor4.3},
$$
\#\{n < 2^t: C_{n} \equiv 2^{k-1} \mod {2^k} \} 
$$
$$= \#\{n < 2^t: d(\, \alpha (\, n )\,  )\, =  k - 1 \}$$
$$
\mbox{$= \binom{t} {k} $  from Theorem~\ref{a}.}
$$
\end{proof}

\bigskip

Theorem~\ref{a} can be used to establish the asymptotic density of the set
$$
\{n < N: C_{n} \equiv 0 \mod {2^k} \}
$$
\bigskip
\begin{thm}
\label{d}
For $k \in \mathbb{N}$ with $k \geq 1$ we have
$$
\lim_{N \to \infty} \frac {1}{N} \#\{n < N: C_{n} \equiv 0 \mod {2^k} \} = 1.
$$
\end{thm}

\begin{proof}
Let $N \in \mathbb{N}$ and choose $r \in \mathbb{N}$ such that
$$
2^{r} \leq N < 2^{r+1}
$$
Then  $r = \lfloor \log_2(\, N)\,  \rfloor$ and
$$
\#\{n < N: d(\, \alpha (\, n )\, )\, = k \} 
$$
$$
\leq \#\{n < 2^{r+1}: d(\, \alpha (\, n )\, )\, = k \}
$$ 
$$
= \binom {r+1} {k + 1}.
$$
So,
$$
\#\{n < N: C_{n} \equiv 0 \mod {2^k} \}
$$ 
$$
= N - \#\{n < N: d(\, \alpha (\, n )\,  )\,  < k \} 
$$
$$
= N - \sum_{i = 0}^{k - 1} \#\{n < N: d(\, \alpha (\, n )\,  )\,  = i \} 
$$
$$
\geq N - \sum_{i = 0}^{k - 1} \binom {r+1} {i + 1}.
$$
And so,
$$
\frac {1}{N} \#\{n < N: C_{n} \equiv 0 \mod {2^k} \} 
$$
$$
\geq 1 - \frac {1}{N} P_{k} (\, r )\,
$$
\begin{equation}
\label{eq1}
= 1 - \frac {1}{N} P_{k} (\, \lfloor \log_2(\, N)\,  \rfloor )\,
\end{equation}
where $P_{k} (\, r )\,$ is a polynomial in $r$ of degree $k$ with coefficients depending on $k$. Since $k$ is fixed and 
$$
\lim_{N \to \infty} \frac {1}{N} (\, \lfloor \log_2(\, N)\,  \rfloor  )\, ^{k} = 0
$$
for fixed $k$, the second term in (\ref{eq1}) is zero in the limit and
$$
\lim_{N \to \infty} \frac {1}{N} \#\{n < N: C_{n} \equiv 0 \mod {2^k} \} = 1.
$$
\end{proof}

\section{Catalan numbers modulo 3}
Deutsch and Sagan \cite{Sagan2006} provided a characterisation of the Catalan numbers modulo~3. They defined a set $T^*(01)$ of natural numbers using the base 3 representation. If $[\,n]\,_{3}=\langle~n_{i} ~\rangle$ is the base 3 representation of $n$ then
$$
\mbox{ $T^{*} (01) = \{n \geq 0: n_i = 0$ or $1$ for all $i \geq 1 \}$ }.
$$
They then defined the function $d_{3}^{*}(\, n \,)$ similar to $d_p(\, n )\,$ in (\ref{d_p}) as 
$$
d_3^*(n) = \#\{n_i: i \geq 1, n_i = 1 \}
$$

\bigskip

\begin{thm}
\label{c3-1}
The asymptotic density of the set $T^*(01)$ is zero.
\end{thm}
\begin{proof}
Let $N \in \mathbb{N}$ and choose $k \in \mathbb{N}: 3^{k} \leq N < 3^{k+1}$. Then $k = \lfloor log_3 (\, N )\, \rfloor$ and

$$
\frac {1}{N} \# \{ n \leq N :  n \in  T^*(01) \}$$  
$$
\leq \frac {1}{N} 3 \times 2^{k}
$$
$$
\mbox{$\leq 3^{1-k} \times 2^{k} \to 0$ as $N \to \infty.$}
$$
\end{proof}

The following theorem comes from \cite{Sagan2006}.

\bigskip

\begin{thm}
\label{c3-2}
(Theorem 5.2 of \cite{Sagan2006}) The Catalan numbers satisfy
$$
\mbox{$C_n~\equiv~(-1)^{d_3^*(n+1)}~\mod~3 $ if $ n~\in~T^*(01)-1$}
$$
$$
\mbox{$C_n~\equiv~0\mod~3$ otherwise.}
$$
\end{thm}

\bigskip

\begin{cor}
$$
\lim_{N \to \infty} \frac {1}{N} \#\{n < N: C_{n} \equiv 0 \mod {3} \} = 1.
$$
\end{cor}
\begin{proof}
Combining theorems \ref{c3-1} and \ref{c3-2}, the set of $n$ such that $C_n$ is not congruent to \mbox{$0~\mod~3$} has asymptotic density $0$. 
\end{proof}

\bigskip

\bibliographystyle{plain}
\begin{small}
\bibliography{ref}
\end{small}

\end{document}